\documentclass[12pt]{amsart}
\usepackage{amsmath,amssymb}
\usepackage{amsaddr}
\usepackage{xurl}
\usepackage{hyperref}
\usepackage[capitalise]{cleveref}
\hypersetup{colorlinks=true,citecolor=red}

\newtheorem{theorem}{Theorem}[section]
\newtheorem{lemma}{Lemma}[section]
\newtheorem{corollary}{Corollary}[section]
\newtheorem{proposition}{Proposition}[section]

\theoremstyle{definition}
\newtheorem{remark}{Remark}[section]
\newtheorem{definition}{Definition}[section]
\newtheorem{example}{Example}[section]

\newcommand\gen[1]{\langle{#1}\rangle}
\newcommand\N{\mathbb{N}}
\newcommand\Z{\mathbb{Z}}
\newcommand\Q{\mathbb{Q}}
\newcommand\doi[1]{DOI: \href{https://doi.org/#1}{#1}}
\DeclareMathOperator{\Nil}{Nil}
\DeclareMathOperator{\Jac}{Jac}
\DeclareMathOperator{\Kdim}{\mathsf{Kdim}}

\begin{document}

\title{A constructive proof of the general Nullstellensatz for Jacobson rings}
\author{Ryota Kuroki}
\email{kuroki-ryota128@g.ecc.u-tokyo.ac.jp}
\address{Graduate School of Mathematical Sciences, The University of Tokyo, 3-8-1 Komaba, Meguro-ku, Tokyo, 153-8914, Japan}
% \date{June 21, 2025}

\begin{abstract}
We give a constructive proof of the general Nullstellensatz: a univariate polynomial ring over a commutative Jacobson ring is Jacobson. This theorem implies that every finitely generated algebra over a zero-dimensional ring or the ring of integers is Jacobson, which has been an open problem in constructive algebra. We also prove a variant of the general Nullstellensatz for finitely Jacobson rings.
\end{abstract}


\maketitle

\section{Introduction}
In this paper, all rings are considered to be commutative with identity.

Constructive algebra is the study of algebra without using non-constructive principles, such as Zorn's lemma or the law of excluded middle.
See \cite{MRR88,LQ15,Yen15,CL24} for recent developments in constructive algebra.
It is closely related to computer algebra since constructive proofs have computational content, which can be extracted using proof assistants, such as Agda, Coq, and Lean.

Despite recent developments in constructive algebra, little work has been done on the constructive theory of Jacobson rings.
In classical mathematics, a Jacobson ring is a ring $A$ such that every prime ideal of $A$ is an intersection of maximal ideals.
One of the most important theorems about Jacobson rings is the general Nullstellensatz: a univariate polynomial ring over a Jacobson ring is Jacobson. This theorem has been proved independently by Goldman \cite[Theorem 3]{Gol51} and Krull \cite[Satz 1]{Kru51}.

Recently, Wessel \cite[Section 2.4.1]{Wes18} has proposed the following constructive definition of a Jacobson ring: a ring $A$ is called Jacobson if $\Nil I=\Jac I$ for all ideals $I$ of $A$, where $\Nil I$ and $\Jac I$ denote the nilradical and Jacobson radical of $I$, respectively.
In this paper, based on this proposal by Wessel and the idea of the classical proof by Emerton \cite[Theorem 8]{Eme}, we give a constructive proof of the general Nullstellensatz (\cref{general-nullstellensatz}).
Our constructive proof provides a solution to two questions \cite{Wer17,Arr21} on MathOverflow and two open problems \cite[{}1.1, 1.2]{Lom23} in Lombardi's list.

We use the entailment relation of prime ideal and the entailment relation of maximal ideal \cite{SW21} to obtain a constructive proof, although familiarity with entailment relations is not necessary to understand our proof.
A notion similar to the entailment relation has been considered by Lorenzen \cite{Lor51}, and Scott \cite{Sco71,Sco74} has later introduced the entailment relation.
The use of entailment relations in algebra dates back to Cederquist and Coquand \cite{CC00} and Coquand and Persson \cite{CP01}. Their work is closely related to the dynamical method by Coste, Lombardi, and Roy \cite{CLR01}, which is a generic method to turn a classical proof into a constructive proof.
Classical arguments involving prime ideals and maximal ideals often use Zorn's lemma.
The dynamical method simulates classical proofs in a simple deductive system, such as entailment relations, and extracts a constructive argument without prime ideals or maximal ideals.

The most important step in our constructive proof of the general Nullstellensatz is the generalization (\cref{Emerton}) of Emerton's key lemma \cite[Lemma 6]{Eme}.
Emerton's lemma is a lemma about integral domains, and we need some non-constructive principle to obtain the general Nullstellensatz from it.
Our generalized lemma does not involve integral domains and is more useful in constructive mathematics.

The definition of a Jacobson ring used in this paper contains a quantification over all ideals of a ring.
This definition works in some predicative foundations of constructive mathematics, such as Martin-Löf type theory with universes using setoids, although we have to pay attention to the universe levels.
On the other hand, this definition does not work in other predicative foundations that cannot handle power sets at all, such as Martin-Löf type theory without universes.
Such foundations do not accept constructing a proposition by quantifying over all subsets of a set.
A definition of a Jacobson ring that does not use this problematic quantification is still not known.

In the last section, we prove a variant of the general Nullstellensatz for finitely Jacobson rings (\cref{fin-general-nullstellensatz}).
A ring $A$ is called finitely Jacobson if $\Nil I=\Jac I$ for all finitely generated ideals $I$ of $A$.
Madden, Delzell, and Ighedo \cite[p.~9]{MDI22} have introduced the notion of finitely Jacobson rings based on the notion of conjunctive lattices \cite[Definition 1]{Sim78}, \cite{DIM21}.
Lombardi \cite[Section 1]{Lom23} has also introduced the same notion under the name of weakly Jacobson rings.
The name comes from the notion of weakly Jacobson lattices \cite[Définitions 1.2]{CLQ06}, which is another name for conjunctive lattices.
Unlike Jacobson rings, finitely Jacobson rings can be defined without the problematic quantification mentioned above.
\section{Jacobson rings}
In this section, we present the constructive definition of Jacobson rings and provide some examples.
We first recall some basic definitions.
\begin{definition}
    Let ${\gen U}_A$ denote the ideal of a ring $A$ generated by a subset $U\subseteq A$. We define two ideals $\Nil_A U$, $\Jac_A U$ of $A$ as follows:
    \begin{gather*}
    \Nil_A U:=\{a\in A:\exists n\ge0.\ a^n\in\gen U\},\\\Jac_AU:=\{a\in A:\forall b\in A.\ 1\in\gen{U,1-ab}\}.
    \end{gather*}
    Here $\gen{U,1-ab}$ means $\gen{U\cup\{1-ab\}}$. When the context is clear, we write $\gen U$, $\Nil U$, and $\Jac U$ for ${\gen U}_A$, $\Nil_AU$, and $\Jac_AU$, respectively.
    Note that $\Nil U=\Nil{\gen{U}}$ and $\Jac U=\Jac{\gen{U}}$ hold for all $U\subseteq A$.
\end{definition}
The ideals $\Nil_A0$ and $\Jac_A0$ are called the \emph{nilradical} and the \emph{Jacobson radical} of $A$, respectively.
Classically, it is well known that the ideal $\Nil U$ (resp. $\Jac U$) is equal to the intersection of all prime (resp. maximal) ideals containing $U$.

The above definition does not use prime ideals or maximal ideals. Therefore, it is reasonable to use the following definition of a Jacobson ring in a constructive setting.
\begin{definition}[{\cite[Section 2.4.1]{Wes18}}]
    We call a ring $A$ \emph{Jacobson} if every ideal $I$ of $A$ satisfies $\Jac I\subseteq\Nil I$.
\end{definition}
In classical mathematics, the above definition is equivalent to the ordinary one.
Note that every subset $U$ of a ring $A$ satisfies $\Nil U\subseteq\Jac U$.
If we constructively prove that a ring $A$ is Jacobson, then we can extract an algorithm such that
\begin{itemize}
    \item its input is an ideal $I$ of $A$, an element $a$ of $A$, and a function $f:A\to I\times A$ such that for any $b\in A$, if $(i,c)=f(b)$, then $1=i+(1-ab)c$, and
    \item its output is a natural number $n\ge0$ such that $a^n\in I$.
\end{itemize}
The following proposition easily follows from the definition.
\begin{proposition}\label{radical-projection}
    Let $I$ be an ideal of a ring $A$. Let $\pi:A\to A/I$ be the canonical projection. For every subset $U$ of $A$, the following equalities hold:
    \[
    \pi^{-1}(\Nil_{A/I}{\pi(U)})=\Nil_A(I\cup U),\quad \pi^{-1}(\Jac_{A/I}{\pi(U)})=\Jac_A(I\cup U).
    \]
\end{proposition}
Two corollaries follow from the above proposition. They are useful for constructing some examples of Jacobson rings.
\begin{corollary}\label{quotient-Jacobson}
    If $I$ is an ideal of a Jacobson ring $A$, then $A/I$ is Jacobson.
\end{corollary}
\begin{corollary}\label{Jacobson-iff-Jac-sub-Nil-for-all-quotients}
    A ring $A$ is Jacobson if and only if $\Jac_{A/I}0\subseteq\Nil_{A/I}0$ holds for all ideals $I$ of $A$.
\end{corollary}
We next prove a fundamental lemma that can be used to avoid the use of prime ideals and maximal ideals.
\begin{lemma}\label{radical-properties}
    Let $A$ be a ring, $U$ be a subset of $A$, and $x,y\in A$. The following statements hold:
    \begin{enumerate}
        \item If $xy\in\Nil U$ and $x\in\Nil{(U,y)}$, then $x\in\Nil U$. Here $\Nil{(U,y)}$ means $\Nil{(U\cup\{y\})}$.
        \item If $x\in\Jac(U,1-yz)$ for all $z\in A$, then $xy\in\Jac U$.
    \end{enumerate}
\end{lemma}
\begin{proof}
    \begin{enumerate}
        \item There exists $n\ge0$ such that $x^n\in\gen U+\gen{y}$. Hence $x^{n+1}\in\gen U+\gen{xy}\subseteq\Nil U$.
        \item For all $w\in A$, we have $1\in\gen{U,1-y(xw),1-x(yw)}=\gen{U,1-(xy)w}$. Hence $xy\in\Jac U$.\qedhere
    \end{enumerate}
\end{proof}
\begin{remark}\label{dynamical}
    This remark is not needed to understand the constructive proof of our main theorem, but it explains the importance of \cref{radical-properties}.
    
    In terms of the entailment relation $\vdash_{\mathrm{p}}$ and the geometric entailment relation $\vdash_{\mathrm{m}}$, which are defined in \cite{SW21}, item 1 of the above lemma corresponds to the fact that $U\vdash_{\mathrm{p}} x,y$ and $U,y\vdash_{\mathrm{p}} x$ together imply that $U\vdash_{\mathrm{p}} x$. Item 2 corresponds to the fact that if $U,1-yz\vdash_{\mathrm{m}}x$ holds for all $z\in A$, then $U\vdash_{\mathrm{m}} x,y$. They are related to the cut rule of $\vdash_{\mathrm{p}}$ and the infinitary cut rule of $\vdash_{\mathrm{m}}$, respectively.

    In classical proofs, we can use results about integral domains to prove something about arbitrary rings.
    We do this by taking the quotient by a prime ideal, and this method does not work constructively, because we do not have enough prime ideals without non-constructive principles such as Zorn's lemma.
    Thus, in constructive algebra, we use the entailment relation $\vdash_{\mathrm{p}}$ to simulate an argument about integral domains, and then extract some results about arbitrary rings.
    We similarly use $\vdash_{\mathrm{m}}$ to extract some general result from an argument about fields.
\end{remark}
\begin{corollary}\label{idempotence}
    Let $A$ be a ring and $U\subseteq A$. Then $\Jac(\Jac U)\subseteq\Jac U$ hold.
\end{corollary}
\begin{proof}
    Let $a\in\Jac(\Jac U)$.
    Then we have $1\in\gen{\Jac U,1-ab}\subseteq\Jac (U,1-ab)$ for every $b\in A$.
    Thus $a\in\Jac U$ by \cref{radical-properties}-(2) with $(A,U,x,y)=(A,U,1,a)$.
\end{proof}
Before we provide some examples of Jacobson rings, we review the basic constructive theory of Krull dimension.
Lombardi \cite[Définition 5.1]{Lom02} has introduced the following constructively acceptable definition of Krull dimension.
\begin{definition}
    Let $n\ge-1$. A ring $A$ is of \emph{Krull dimension at most $n$} if for every $x_0,\ldots,x_n\in A$, there exists $e_0,\ldots,e_n\ge0$ such that
    \[x_0^{e_0}\cdots x_n^{e_n}\in\gen{x_0^{e_0+1},x_0^{e_0}x_1^{e_1+1},\ldots,x_0^{e_0}\cdots x_{n-1}^{e_{n-1}}x_n^{e_n+1}}.\]
    Let $\Kdim A\le n$ denote the statement that $A$ is of Krull dimension at most $n$.
    A ring $A$ is called \emph{$n$-dimensional} if $\Kdim A\le n$.
\end{definition}
\begin{example}
    Every discrete field is zero-dimensional, where a discrete field means a ring $K$ such that every element of $K$ is null or invertible.
    We also have $\Kdim\Z\le1$ \cite[Examples below Lemma XIII-2.4]{LQ15}, \cite[Examples 86]{Yen15}.
\end{example}
\begin{proposition}[{\cite[Proposition XIII-3.1, Lemma IX-1.2]{LQ15}}]\label{Kdim-proposition}
    Let $A$ be a ring with $\Kdim A\le n$. The following statements hold:
    \begin{enumerate}
        \item If $I$ is an ideal of $A$, then $\Kdim A/I\le n$.
        \item If $n\ge0$ and $I$ is an ideal of $A$ containing a regular element $x$, then $\Kdim A/I\le n-1$.
        \item If $n=0$, then $\Jac_A0\subseteq\Nil_A0$.
    \end{enumerate}
\end{proposition}
Now, we can show that every zero-dimensional ring is Jacobson, which is essentially contained in \cite[Lemma IX-1.2]{LQ15}.
\begin{example}\label{zero-dim-implies-Jacobson}
    Let $A$ be a zero-dimensional ring.
    By \cref{Kdim-proposition}, we have $\Jac_{A/I}0\subseteq\Nil_{A/I}0$ for all ideals $I$ of $A$.
    Thus, $A$ is Jacobson by \cref{Jacobson-iff-Jac-sub-Nil-for-all-quotients}. In particular, every discrete field is Jacobson.
\end{example}
We next prove that the ring of integers $\Z$ is Jacobson.
Lombardi \cite{Lom24} has pointed out that the following lemma generalizes the essential part of our original proof that $\Z$ is Jacobson.
\begin{lemma}\label{Kdim1-Jacobson-lemma}
    Let $A$ be a $1$-dimensional ring. If an ideal $I$ of $A$ contains some regular element, then $\Jac I\subseteq\Nil I$.
\end{lemma}
\begin{proof}
    We have $\Kdim A/I\le0$ by \cref{Kdim-proposition}. Hence we have $\Jac_{A/I}0\subseteq\Nil_{A/I}0$ by \cref{zero-dim-implies-Jacobson}. Hence, the assertion follows from \cref{radical-projection}.
\end{proof}
The following proposition gives a class of Jacobson rings that contains $\Z$.
\begin{proposition}\label{Kdim1-Jacobson}
    If $A$ is a ring satisfying all of the following conditions, then $A$ is Jacobson:
    \begin{enumerate}
        \item Every element of $A$ is nilpotent or regular.
        \item The Krull dimension of $A$ is at most 1.
        \item For all ideals $I$ of $A$, if $\Jac I$ contains some regular element, then $I$ also contains some regular element.
    \end{enumerate}
    In particular, $\Z$ is Jacobson.
\end{proposition}
\begin{proof}
    Let $a\in\Jac I$.
    \begin{enumerate}
        \item If $a$ is nilpotent, then $a\in\Nil I$.
        \item If $a$ is regular, then $I$ contains some regular element. Hence $a\in\Nil I$ by \cref{Kdim1-Jacobson-lemma}.\qedhere
    \end{enumerate}
\end{proof}
Assuming the law of excluded middle, we can replace the third condition of the above proposition with $\Jac0\subseteq\Nil0$ by the following proposition:
\begin{proposition}\label{simplify-Kdim1-Jacobson}
    Let $A$ be a ring satisfying the following two conditions:
    \begin{enumerate}
        \item Every element of $A$ is nilpotent or regular.
        \item $\Jac0\subseteq\Nil0$.
    \end{enumerate}
    Then, for all ideals $I$ of $A$, if $\Jac I$ contains some regular element, then the double negation of ``$I$ contains a regular element'' holds.
\end{proposition}
\begin{proof}
    Let $I$ be an ideal of $A$ such that $\Jac I$ contains some regular element.
    \begin{itemize}
        \item Assume that $I$ does not contain a regular element. Then $I\subseteq\Nil 0\subseteq\Jac0$.
        Hence $\Jac I\subseteq\Jac(\Jac 0)\subseteq\Jac 0\subseteq\Nil0$ by \cref{idempotence}.
        Since $\Jac I$ contains some regular element, $A$ has a regular and nilpotent element. Hence $0\in I$ is regular, and this contradicts the assumption that $I$ does not contain a regular element.
    \end{itemize}
    Thus, the double negation of ``$I$ contains a regular element'' holds.
\end{proof}
In \cref{Z-is-Jacobson}, we use the following lemma to present the computational content of $\Z$ being a Jacobson ring.
The proof of the lemma is similar to the proof of $1$-dimensionality of $\Z$ in \cite[Examples 86]{Yen15}.
\begin{lemma}\label{Z-is-Jacobson-lemma}
    Let $x\in\Z-\{0\}$ and $a\in\Z$. Then, there exist $d,e\in\Z$ such that $x=de$, $a\in\Nil d$, and $1\in\gen{a,e}$.
\end{lemma}
\begin{proof}
    Computing successively \[d_1=\gcd(x,a),\,d_2=\gcd(\frac{x}{d_1},a),\,d_3=\gcd(\frac{x}{d_1d_2},a),\ldots,
    \]we obtain a finite sequence $(d_1,\dots,d_n)$ of positive integers such that $d_i\mid a$ and $d_n=1$.
    By letting $d:=d_1\cdots d_{n-1}$ and $e:=x/d$, we have $x=de$, $a\in\Nil d$, and $1\in\gen{a,e}$.
\end{proof}
\begin{remark}\label{Z-is-Jacobson}
    The ring of integers $\Z$ is Jacobson by \cref{Kdim1-Jacobson}. We can unfold the constructive proof and reveal the computational content as follows.
    Let $I$ be an ideal of $\Z$ and $a\in\Jac I$.
    \begin{enumerate}
        \item If $a=0$, then $a\in\Nil I$.
        \item If $a\ge1$, then there exists $c_{-1}\in\Z$ such that $1-(1+a)c_{-1}\in I$ by $a\in\Jac I$. Since $1-(1+a)c_{-1}\ne0$, we deduce from the \cref{Z-is-Jacobson-lemma} that there exist $d,e\in\Z$ such that $1-(1+a)c_{-1}=de$, $a\in\Nil d$, and $1\in\gen{a,e}$ all hold. By $1\in\gen{a,e}$, there exists $b\in\Z$ such that $1-ab\in\gen e$. By $a\in\Jac I$, there exists $c_b\in\Z$ such that $1-(1-ab)c_b\in I$. Since $d\in\gen{1-(1+a)c_{-1},1-(1-ab)c_b}\subseteq I$, we have $a\in\Nil d\subseteq\Nil I$.
        \item If $a\le-1$, then $a\in\Nil I$ by a similar argument.
    \end{enumerate}
    This solves a problem that Werner has mentioned in his comment to his question \cite{Wer17} on MathOverflow.
    Note that in the above argument, the ideal $I\subseteq\Z$ is not assumed to be finitely generated or detachable, where a detachable ideal of a ring $A$ means an ideal $I\subseteq A$ such that $x\in I$ or $x\notin I$ holds for all $x\in A$.
\end{remark}
\section{The general Nullstellensatz}
In this section, we provide a constructive proof of the general Nullstellensatz.
We first recall three basic constructive results.
\begin{proposition}[{\cite[Corollary VI-1.3]{MRR88}}]\label{integrality-monogenic-algebra}
    Let $B$ be an $A$-algebra. If $b\in B$ is integral over $A$, then the algebra $A[b]\subseteq B$ is integral over $A$.
\end{proposition}
\begin{proposition}[{\cite[Theorem IX-1.7]{LQ15}}]\label{invertibility-integral-extension}
    If $B$ is an integral extension of a ring $A$, then $B^\times\cap A\subseteq A^\times$.
\end{proposition}
\begin{proposition}[{\cite[Lemma II-2.6]{LQ15}}]\label{invertible-polynomial}
    For all rings $A$, $A[X]^\times\subseteq A^\times+(\Nil_A0)[X]$.
\end{proposition}
\begin{corollary}[Snapper's theorem {\cite[Corollary 8.1]{Sna50}}]\label{polynomial-jac-nil}
    For all rings $A$, $\Jac_{A[X]}0\subseteq\Nil_{A[X]}0$.
\end{corollary}
\begin{proof}
    Let $f\in\Jac_{A[X]}0$. Then we have $1-Xf\in A[X]^{\times}$. Hence $f\in(\Nil_A0)[X]$ by \cref{invertible-polynomial}, and thus $f\in\Nil_{A[X]}0$.
\end{proof}
For $a\in A$, let $A_a$ denote the localization $A[1/a]$. Quitté \cite{Qui24} has pointed out that the following lemma is essentially contained in our original proof of \cref{Emerton}.
\begin{lemma}\label{quitte-lemma}
    Let $A$ be a ring and $a\in A$. Let $B$ be a ring extension of $A$ such that $B_a$ is integral over $A_a$. Then $a((\Jac_B0)\cap A)\subseteq\Jac_A0$.
\end{lemma}
\begin{proof}
    Let $a'\in(\Jac_B0)\cap A$. For all $x\in A$, we have $1-a'x\in B^\times$. Hence $1-a'x\in B_a^\times\cap A_a\subseteq A_a^\times$ by \cref{invertibility-integral-extension}.
    Hence $a\in\Nil_A(1-a'x)\subseteq\Jac_A(1-a'x)$. Thus, $aa'\in\Jac_A0$ by \cref{radical-properties}-(2) with $(A,U,x,y)=(A,0,a,a')$.
\end{proof}
The following lemma is a constructive counterpart of the key lemma given by Emerton \cite[Lemma 6]{Eme}.
Emerton's lemma assumes that $A$ and $B$ are integral domains.
As explained in \cref{dynamical}, we need to remove this assumption to eliminate the use of prime ideals and maximal ideals and to obtain a constructive proof of the general Nullstellensatz.
Generalizing Emerton's key lemma is the most crucial process in the proof of our main theorem.
\begin{lemma}\label{Emerton}
    Let $A$ be a Jacobson ring, $a\in A$, and $B$ be an $A$-algebra such that $B_a$ is integral over $A_a$.
    Then $a\Jac_BJ\subseteq\Nil_BJ$ holds for all ideals $J$ of $B$.
\end{lemma}
\begin{proof}
    Let $B':=B/J$ and $\varphi:A\to B'$ be the canonical homomorphism.
    Let $A':=A/{\ker{\varphi}}$. The ring $B'$ is a ring extension of $A'$.
    Let $f\in\Jac_BJ$.
    Since $B'_a$ is integral over $A'_a$, there exist $n,l\ge0$ and $a_0,\ldots,a_{n-1}\in A'$ such that $a^lf^n+a_{n-1}f^{n-1}+\cdots+a_0=_{B'}0$.
    For $k\in\{0,\ldots,n\}$, let $g_k:=a^lf^{k}+a_{n-1}f^{k-1}+\cdots+a_{n-k}$, $J_k:=\gen{g_k,\ldots,g_{n-1}}_{B'}$, and $(A_k,B_k):=(A'/(J_k\cap A'),B'/J_k)$.
    Note that $B_k$ is a ring extension of $A_k$ such that $(B_k)_a$ is integral over $(A_k)_a$.
    We prove that $af\in\Nil_{B_k}0$ for all $k$ by induction.
    \begin{enumerate}
        \item Since $a^l=_{A_0}0$, we have $af\in\Nil_{B_0}0$.
        \item
            Let $k\ge1$. We have $a_{n-k}\in\Jac_{B_k}0$ by $f,g_k\in\Jac_{B_k}0$.
            Hence $aa_{n-k}\in\Jac_{A_k}0\subseteq\Nil_{A_k}0\subseteq\Nil_{B_k}0$ by \cref{quitte-lemma} and the Jacobsonness of $A_k$. Hence $afg_{k-1}\in\Nil_{B_k}0$ by $a_{n-k}=g_k-fg_{k-1}=_{B_k}-fg_{k-1}$.
                We have $af\in\Nil_{B_{k-1}}0$ by the inductive hypothesis.
                Hence $af\in\Nil_{B_k}g_{k-1}$ by \cref{radical-projection}. Thus, $af\in\Nil_{B_k}0$ by \cref{radical-properties}-(1) with $(A,U,x,y)=(B_k,0,af,g_{k-1})$.
    \end{enumerate}
    Thus, $af\in\Nil_{B_n}0=\Nil_{B'}0$, and hence $af\in\Nil_BJ$.
\end{proof}
The above key lemma implies the following two corollaries, which are useful for constructing examples of Jacobson rings.
\begin{corollary}[pointed out by Quitté {\cite{Qui24}}]
    If $A$ is a Jacobson ring and $a\in A$, then $A_a$ is Jacobson.
\end{corollary}
\begin{proof}
    Let $B=A_a$ in \cref{Emerton}.
\end{proof}
\begin{corollary}
    If $B$ is an algebra integral over a Jacobson ring $A$, then $B$ is Jacobson.
\end{corollary}
\begin{proof}
    Let $a=1$ in \cref{Emerton}.
\end{proof}
We are now ready to prove our main theorem.
Quitté \cite{Qui24} has simplified our original proof by removing redundant inductive hypotheses.
\begin{theorem}[The General Nullstellensatz]\label{general-nullstellensatz}
    If $A$ is a Jacobson ring, then so is $A[X]$.
\end{theorem}
\begin{proof}
    Let $J$ be an ideal of $A[X]$ and $f\in\Jac_{A[X]}J$.
    Since $1\in\gen{J,1-fX}_{A[X]}$, there exists $g\in J$ such that $1\in\gen{g,1-fX}_{A[X]}$.
    There exist $n\ge0$ and $a_0,\ldots,a_n\in A$ such that $g=a_nX^n+\cdots+a_0$.

    Let $C_k:=A[X]/\gen{J,a_{k+1},\ldots,a_n}$ for $k\in\{-1,\ldots,n\}$.
    We prove that $f\in\Nil_{C_k}0$ for all $k$ by induction.
    \begin{enumerate}
        \item Let $A':=A/\gen{a_0,\ldots,a_n}$. We have $f\in\Nil_{A'[X]}0$ by $1\in\gen{1-fX}_{A'[X]}$ and \cref{invertible-polynomial}. Hence $f\in\Nil_{C_{-1}}0$.
        \item Let $k\ge0$. 
        Then $X\in (C_k)_{a_k}$ is integral over $A_{a_k}$ since $g=_{(C_k)_{a_k}}0$.
        Hence, $(C_k)_{a_k}$ is integral over $A_{a_k}$ by \cref{integrality-monogenic-algebra}. Since $f\in\Jac_{C_k}0$, we have $a_kf\in\Nil_{C_k}0$ by \cref{Emerton}.
        We have $f\in\Nil_{C_{k-1}}0$ by the inductive hypothesis. Hence $f\in\Nil_{C_{k}}a_k$. Thus, $f\in\Nil_{C_k}0$ by \cref{radical-properties}-(1) with $(A,U,x,y)=(C_k,0,f,a_k)$.
    \end{enumerate}
    Thus, $f\in\Nil_{C_n}0=\Nil_{A[X]/J}0$, and hence $f\in\Nil_{A[X]}J$.
\end{proof}
The main theorem and \cref{quotient-Jacobson} together imply the following corollary.
\begin{corollary}
    Every finitely generated algebra over a Jacobson ring is Jacobson.
\end{corollary}
We obtain a solution to two problems \cite[{}1.1, 1.2]{Lom23} in Lombardi's list by \cref{zero-dim-implies-Jacobson}, \cref{Z-is-Jacobson}, and the above corollary.
\begin{corollary}
    Let $A$ be a zero-dimensional ring or the ring of integers $\Z$.
    Then every finitely generated algebra over $A$ is Jacobson.
\end{corollary}
\section{Finitely Jacobson rings}
In this section, we develop the theory of finitely Jacobson rings and prove a variant of the general Nullstellensatz.

The following definition of finitely Jacobson rings has been proposed by Madden, Delzell, and Ighedo \cite[p.~9]{MDI22}.
Lombardi \cite[Section 1]{Lom23} has also introduced the same notion under the name of weakly Jacobson rings.
\begin{definition}
    We call a ring $A$ \emph{finitely Jacobson} if every finitely generated ideal $I$ of $A$ satisfies $\Jac I\subseteq\Nil I$.
\end{definition}
Every Jacobson ring is finitely Jacobson.
There is a finitely Jacobson ring that is not Jacobson:
\begin{proposition}
    Let $A$ be the polynomial ring $\Q[X_k:k\in\N]$ in countably infinitely many variables over $\Q$.
    Then $A$ is finitely Jacobson but not Jacobson.
\end{proposition}
\begin{proof}
    \begin{enumerate}
        \item We first prove that $A$ is finitely Jacobson. Let $f_1,\ldots,f_n\in A$ and $I:=\gen{f_1,\ldots,f_n}$. Then, there exists $m\in\N$ such that $f_1,\ldots,f_n\in A_m$, where $A_m:=\Q[X_k:k\in\N-\{m\}]$.
        Since $A=A_m[X_m]$ and $A/I=(A_m/\gen{f_1,\ldots,f_n})[X_m]$, we have $\Jac_{A/I}0\subseteq\Nil_{A/I}0$ by \cref{polynomial-jac-nil}.
        Hence $\Jac_A I=\Nil_A I$ by \cref{radical-projection}.
        \item 
        We next prove that $A$ is not Jacobson.
        Let $S:=1+X\Q[X]$ and $B:=S^{-1}(\Q[X])\subseteq\Q(X)$.
        We have $X\in\Jac_B0$, since for every $f\in\Q[X]$ and $g\in S$, we have $g-Xf\in S$ and $1-X(f/g)=(g-Xf)/g\in B^\times$.
        We also have $X\notin\Nil_B0$ since $\Nil_B0=0$.
        Hence, $B$ is not Jacobson.
        Since $B$ is a countable $\Q$-algebra, there exists a surjective ring homomorphism from $A$ to $B$. Hence $A$ is not Jacobson by \cref{quotient-Jacobson}.\qedhere
    \end{enumerate}
\end{proof}
The following proposition corresponds to \cref{quotient-Jacobson}.
\begin{proposition}\label{fin-quotient-Jacobson}
    If $I$ is a finitely generated ideal of a finitely Jacobson ring $A$, then $A/I$ is finitely Jacobson.
\end{proposition}
\begin{proof}
    Let $J$ be a finitely generated ideal of $A/I$.
    There exist $a_1,\ldots,a_m,b_1,\ldots,b_n\in A$ such that $I=\gen{a_1,\ldots,a_m}_A$ and $J=\gen{b_1,\ldots,b_n}_{A/I}$.
    Let $\pi:A\to A/I$ denote the canonical projection.
    By \cref{radical-projection}, we have $\Jac J=\pi(\Jac_A \gen{a_1,\ldots,a_m,b_1,\ldots,b_n})\subseteq\pi(\Nil_A \gen{a_1,\ldots,a_m,b_1,\ldots,b_n})=\Nil J$.
\end{proof}
The following proposition corresponds to \cref{Kdim1-Jacobson} and \cref{simplify-Kdim1-Jacobson}.
Note that we do not need to assume the law of excluded middle to obtain the following result:
\begin{proposition}\label{fin-Kdim1-Jacobson}
    If $A$ is a ring satisfying all of the following conditions, then $A$ is Jacobson:
    \begin{enumerate}
        \item Every element of $A$ is nilpotent or regular.
        \item The Krull dimension of $A$ is at most 1.
        \item $\Jac0\subseteq\Nil0$.
    \end{enumerate}
\end{proposition}
\begin{proof}
    Using a similar argument as in the proof of \cref{simplify-Kdim1-Jacobson}, we first prove that for any finitely generated ideals $I=\gen{a_1,\ldots,a_n}$ of $A$, if $\Jac I$ contains a regular element, then $I$ also contains a regular element.
    \begin{enumerate}
        \item If there exists $i\in\{1,\ldots,n\}$ such that $a_i$ is regular, then $I$ contains a regular element.
        \item If $a_1,\ldots,a_n$ are all nilpotent, then $I\subseteq \Nil0\subseteq\Jac0$. Hence $\Jac I\subseteq \Jac(\Jac0)\subseteq\Jac0\subseteq\Nil0$ by \cref{idempotence}.
        Since $\Jac I$ contains a regular element, $A$ has a regular and nilpotent element. Hence $0\in I$ is regular.
    \end{enumerate}
    The rest of the proof is similar to the proof of \cref{Kdim1-Jacobson}.
\end{proof}
We introduce the notion of good algebras to prove a variant of the general Nullstellensatz.
\begin{definition}
    Let $A$ be a ring. We call an $A$-algebra $(B,\varphi:A\to B)$ \emph{good} if for all finitely generated ideal $J$ of $B$, the ideal $\varphi^{-1}(J)\subseteq A$ is finitely generated.
\end{definition}
\begin{proposition}\label{fin-quotient-good}
    Let $(B,\varphi:A\to B)$ be a good $A$-algebra.
    \begin{enumerate}
        \item Let $A':=A/{\ker\varphi}$. Then $B$ is a good $A'$-algebra.
        \item Let $J$ be a finitely generated ideal of $B$. Then $B/J$ is a good $A$-algebra.
    \end{enumerate}
\end{proposition}
\begin{proof}
    \begin{enumerate}
        \item Let $\pi:A\to A'$ be the canonical projection and $\overline\varphi:A'\to B$ be the canonical homomorphism.
        If $J$ is a finitely generated ideal of $B$, then ${\overline\varphi}^{-1}(J)=\pi(\varphi^{-1}(J))$ is finitely generated.
        \item Since $B/J$ is a good $B$-algebra, it is also a good $A$-algebra.\qedhere
    \end{enumerate}
\end{proof}
The following lemma corresponds to \cref{Emerton}.
\begin{lemma}\label{fin-Emerton}
    Let $A$ be a finitely Jacobson ring and $a\in A$. Let $B$ be a good $A$-algebra such that $B_a$ is integral over $A_a$.
    Then $a\Jac_BJ\subseteq\Nil_BJ$ holds for all finitely generated ideals $J$ of $B$.
\end{lemma}
\begin{proof}
    By noting the following point, we can prove this similarly to the proof of \cref{Emerton}:
    \begin{itemize}
        \item Since $J$ is finitely generated, the ring $B':=B/J$ in the proof is a good $A'$-algebra by \cref{fin-quotient-good}.
        Hence $J_k\cap A'$ is finitely generated, and $A_k$ is finitely Jacobson by \cref{fin-quotient-Jacobson}.\qedhere
    \end{itemize}
\end{proof}
Following \cite{HP89}, we define the contraction property of a ring.
\begin{definition}
    A ring $A$ satisfies the \emph{contraction property (CP)} if $A[X]$ is a good $A$-algebra.
\end{definition}
\begin{theorem}\label{fin-general-nullstellensatz}
    Let $A$ be a finitely Jacobson ring satisfying CP. Then $A[X]$ is finitely Jacobson.
\end{theorem}
\begin{proof}
    By noting the following point, we can prove this similarly to the proof of \cref{general-nullstellensatz}.
    \begin{itemize}
        \item Let $J$ be a finitely generated ideal of $A[X]$. Then, $C_k:=A[X]/\gen{J,a_{k+1}.\ldots,a_n}$ in the proof is a good $A$-algebra by \cref{fin-quotient-good}.
        Hence, we can use \cref{fin-Emerton} instead of \cref{Emerton}.
        \qedhere
    \end{itemize}
\end{proof}
The author does not know whether there exists a finitely Jacobson ring $A$ such that $A[X]$ is not finitely Jacobson.

There are several known classes of rings satisfying CP.
For example, every coherent Noetherian ring (in the sense of Richman and Seidenberg \cite{Ric74,Sei74}) satisfies CP (\cite[Corollary 1 of Lemma 4]{Sei74}, \cite[Theorem VIII-1.2]{MRR88}).
Hence, we have the following corollary:
\begin{corollary}
    If $A$ is a coherent Noetherian finitely Jacobson ring, then $A[X_1,\ldots,X_n]$ is finitely Jacobson.
\end{corollary}
\begin{proof}
    By Hilbert's basis theorem (\cite[Theorem 1]{Sei74}, \cite[Theorem VIII-1.5]{MRR88}), the ring $A[X_1,\ldots,X_k]$ is coherent Noetherian for every $k\in\{0,\ldots,n-1\}$. Hence, the assertion follows from \cref{fin-general-nullstellensatz}.
\end{proof}
One-dimensional Pr\"ufer domains also satisfy CP. This is classically proved in \cite[Theorem 2.5]{HP89}, and as mentioned in \cite[Example 4]{PY14}, it can be proved constructively using a result for valuation domains \cite{LSY12} and dynamical Gröbner bases \cite{KY10,Yen06}.
The constructive proof shows that they are furthermore $1$-Gröbner rings.
It is also known that they are Gröbner rings \cite[Corollary 6]{Yen14}.

\section*{Acknowledgments}
The author would like to express his deepest gratitude to his supervisor, Ryu Hasegawa, for his support.
The author would like to thank Claude Quitté, Henri Lombardi, Daniel Misselbeck-Wessel, Ryotaro Iwane, and Yuto Ikeda for their helpful advice.
The author would also like to thank Ralph Willox for his writing advice.
Finally, the author would like to thank the anonymous referee for many useful comments.

This research was supported by Forefront Physics and Mathematics Program to Drive Transformation (FoPM), a World-leading Innovative Graduate Study (WINGS) Program, the University of Tokyo.


\begin{thebibliography}{99}
\bibitem[Arr21]{Arr21} Arrow (https://mathoverflow.net/users/69037/arrow) (2021). \emph{Can you constructively prove a univariate polynomial algebra over a Jacobson ring is itself Jacobson?}. MathOverflow. (version: 2021-11-06). Available at: \url{https://mathoverflow.net/q/405685}.
\bibitem[CC00]{CC00} Cederquist, J., Coquand, T. (2000). Entailment relations and distributive lattices. In: Buss, S.~R., Hájek, P., Pudlák, P., eds. \emph{Logic Colloquium ’98. Proceedings of the Annual European Summer Meeting of the Association for Symbolic Logic, Prague, Czech Republic, August 9–15, 1998}.
Lecture Notes in Logic, 13.
Natick, MA: A K Peters,
pp.~127--139.
\bibitem[CL24]{CL24} Coquand, T., Lombardi, H. (2024). \emph{Résolutions libres finies. Méthodes constructives}. Paris, France: Calvage et Mounet.
\bibitem[CLQ06]{CLQ06} Coquand T., Lombardi H., Quitte C. (2006). Dimension de Heitmann des treillis distributifs et des anneaux commutatifs. \emph{Publications mathématiques de Besançon. Algèbre et Théorie des Nombres}, pp. 57--100.
\bibitem[CP01]{CP01} Coquand, T., Persson, H. (2001). Valuations and Dedekind’s Prague theorem. \emph{J. Pure Appl. Algebra} 155(2--3):121--129. 
\bibitem[CLR01]{CLR01} Coste, M., Lombardi, H., Roy, M.~F. (2001). Dynamical method in algebra: effective Nullstellens\"atze. \emph{Ann. Pure Appl. Logic}, 111(3):203--256.
\bibitem[DIM21]{DIM21} Delzell, C.~N., Ighedo, O., Madden, J.~J. (2021). Conjunctive join-semilattices. \emph{Algebra Univers.} 82, 51.
\bibitem[Eme]{Eme} Emerton, M. \emph{Jacobson rings}. Available at: \url{https://www.math.uchicago.edu/~emerton/pdffiles/jacobson.pdf}.
\bibitem[Gol51]{Gol51} Goldman, O. (1951). Hilbert rings and the Hilbert Nullstellensatz. \emph{Math. Z.} 54(2):136--140.
\bibitem[HP89]{HP89} Heinzer, W.~J., Papick, I.~J. (1989). Remarks on a remark of Kaplansky. \emph{Proc. Amer. Math. Soc.} 105:1--9.
\bibitem[KY10]{KY10} Kacem, A.~H., Yengui, I. (2010). Dynamical Gröbner bases over Dedekind rings. \emph{J. Algebra}. 324(1):12--24.
\bibitem[Kru51]{Kru51} Krull, W. (1951). Jacobsonsche Ringe, Hilbertscher Nullstellensatz, Dimensionstheorie. \emph{Math. Z.} 54(4):354--387.
\bibitem[Lom02]{Lom02} Lombardi, H. (2002). Dimension de Krull, Nullstellens\"atze et évaluation dynamique. \emph{Math. Z.} 242(1):23--46.
\bibitem[Lom23]{Lom23} Lombardi, H. (2023). \emph{Some classical results in Algebra needing one or several constructive versions}. Available at: \url{https://groups.google.com/g/constructivenews/c/Z6ZEmRdep8o/m/TNVpuihzAAAJ}.
\bibitem[Lom24]{Lom24} Lombardi, H. Personal communication, 2024.
\bibitem[LQ15]{LQ15} Lombardi, H., Quitté, C. (2015). \emph{Commutative algebra: constructive methods}. (Roblot, T.~K., trans.) Algebra and Applications, Vol. 20. Dordrecht, The Netherlands: Springer.
\bibitem[LSY12]{LSY12} Lombardi, H., Schuster, P., Yengui, I. (2012). The Gröbner ring conjecture in one
variable, \emph{Math. Z.} 270(3--4):1181--1185.
\bibitem[Lor51]{Lor51} Lorenzen, P. (1951). Algebraische und logistische Untersuchungen über freie Verbände. \emph{Journal of Symbolic Logic}. 16(2):81--106.
\bibitem[MDI22]{MDI22} Madden, J.~J., Delzell, C.~N., Ighedo, O. (2022). Recent Work on Conjunctive Join-Semilattices. Presented at BLAST 2022, Orange, CA, August 9. Available at: \url{https://math.chapman.edu/blast2022/Tuesday/Madden-BLAST2022.pdf}
\bibitem[MRR88]{MRR88} Mines, R., Richman, F., Ruitenburg, W. (1988). \emph{A course in constructive algebra}. Universitext. New York, NY: Springer.
\bibitem[PY14]{PY14} Pola, E., Yengui, I. (2014). Gröbner rings. \emph{Acta Sci. Math.} 80(3--4):363--372.
\bibitem[Qui24]{Qui24} Quitté, C. Personal communication, 2024
\bibitem[Ric74]{Ric74} Richman, F. (1974). Constructive aspects of Noetherian rings. \emph{Proc. Amer. Mat. Soc.} 44(2):436--441.
\bibitem[Sim78]{Sim78} Simmons, H. (1978). The lattice theoretic part of topological separation properties. \emph{Proceedings of the Edinburgh Mathematical Society}. 21(1):41--48.
\bibitem[SW21]{SW21} Schuster, P., Wessel, D. (2021). Syntax for semantics: Krull's maximal ideal theorem. In: Heinzmann, G., Wolters, G., eds. \emph{Paul Lorenzen -- mathematician and logician}. Logic, Epistemology, and the Unity of Science. Vol.~51. Cham, Switzerland: Springer, pp.~77--102.
\bibitem[Sco71]{Sco71} Scott, D. (1971). On engendering an illusion of understanding. \emph{Journal of Philosophy}. 68(21):787--807.
\bibitem[Sco74]{Sco74} Scott, D. (1974). Completeness and axiomatizability in many-valued logic. In: Henkin, L., Addison, J., Chang, C.~C., Craig, W., Scott, D., Vaught, R., eds. \emph{Proceedings of the Tarski Symposium (Proc. Sympos. Pure Math., Vol. XXV, Univ. California, Berkeley, Calif., 1971)}, Providence, RI: American Mathematical Society, pp.~411--435.
\bibitem[Sei74]{Sei74} Seidenberg, A. (1974). What is Noetherian? \emph{Rend. Sem. Mat. e Fis. Milano}. 44:55--61.
\bibitem[Sna50]{Sna50} Snapper, E. (1950). Completely primary rings. I. \emph{Ann. of Math.} 52(3):666--693.
\bibitem[Wer17]{Wer17} Werner, J. (https://mathoverflow.net/users/112369/jakob-werner) (2017). \emph{Constructive treatment of Jacobson rings}. MathOverflow. (version: 2017-07-18). Available at: \url{https://mathoverflow.net/q/275737}.
\bibitem[Wes18]{Wes18} Wessel, D. (2018). \emph{Choice, extension, conservation. From transfinite to finite proof methods in abstract algebra}. PhD thesis. University of Trento, University of Verona.
\bibitem[Yen06]{Yen06} Yengui, I. (2006). Dynamical Gröbner bases. \emph{J. Algebra}. 301(2):447--458.
\bibitem[Yen14]{Yen14} Yengui, I. (2014). The Gröbner Ring Conjecture in the lexicographic order case. \emph{Math. Z.} 276(1--2):216--265.
\bibitem[Yen15]{Yen15} Yengui, I. (2015). \emph{Constructive commutative algebra}. Lecture Notes in Mathematics. Vol.~2138. Cham, Switzerland: Springer.
\end{thebibliography}
\end{document}